\documentclass{amsart}

\usepackage{amsmath, amssymb, 
graphicx, mathrsfs, enumerate}
\usepackage[all]{xy}
\usepackage{color}
\usepackage{latexsym}

\newtheorem{proposition}{Proposition}[section]
\newtheorem{lemma}[proposition]{Lemma}
\newtheorem{corollary}[proposition]{Corollary}
\newtheorem{theorem}[proposition]{Theorem}
\newtheorem{definition}[proposition]{Definition}

\begin{document}
	
	\title{On the Gowers trick for classical simple groups}
	
	\author{Francesco Fumagalli} 
	\address{Dipartimento di Matematica  e
      Informatica 'Ulisse Dini', 
		Viale Morgagni 67/A, 50134 Firenze, 
      Italy}
	\email{francesco.fumagalli@unifi.it}
	
	\author{Attila Mar\'oti}
	\address{Hun-Ren Alfr\'ed R\'enyi Institute of 
 Mathematics, Re\'altanoda utca 13-15, H-1053, 
 Budapest, Hungary}
	\email{maroti.attila@renyi.hu}

\keywords{Gowers trick, simple group, character bounds}
\subjclass[2020]{20D06, 20D40, 20G05}
\thanks{The first author is partially supported
by the Italian INdAM-GNSAGA. 
The second author thanks the hospitality of the Department of Mathematics at the University of Florence. He was also supported by the National Research, 
Development and Innovation Office (NKFIH) Grant 
No.~K138596, No.~K132951 and Grant No.~K138828.}
\maketitle

\begin{abstract} 
If $A$, $B$, $C$ are subsets in  a  finite simple group of Lie type $G$ at least two of which are normal with  
$|A||B||C|$ relatively large, then we establish a stronger conclusion than $ABC = G$. This is related to a theorem of Gowers and is a generalization of a theorem of Larsen, Shalev, Tiep and the second author and Pyber.  
\end{abstract}

\section{Introduction}

Let $A$, $B$, $C$ be subsets of a finite group 
$G$. Let $\mathrm{Prob}(A,B,C)$ be the 
probability that if $a$ and $b$ are uniformly 
and randomly chosen elements from $A$ and $B$ 
respectively, then $ab \in C$. Recall that a subset of $G$ is normal if it is invariant under conjugation by every element of $G$. 

\begin{theorem}\label{main}	
There exists a universal constant $\delta > 0$ 
such that whenever $G$ is a finite simple group of Lie type and whenever $A$, $B$, $C$ are subsets in 
$G$ such that
\begin{enumerate}
\item at least two of the three subsets $A,B,C$ are normal in $G$ and 
\item $|A||B||C| > |G|^{3-\delta}/\eta^2$  for some given $\eta$ with  $0<\eta<1/4$,
\end{enumerate}
then 
$$(1 - \eta) \frac{|C|}{|G|} < \mathrm{Prob}
(A,B,C) < (1 + \eta) \frac{|C|}{|G|}$$ and, for 
any $g \in G$, the number $N$ of triples 
$(a,b,c) \in A \times B \times C$ such that $abc 
= g$ satisfies $$(1 - \eta) 
\frac{|A||B||C|}{|G|} < N < (1 + \eta) 
\frac{|A||B||C|}{|G|}.$$ 	
\end{theorem}

Larsen, Shalev, Tiep \cite[Theorem 7.4]{LST} and 
the second author and Pyber \cite[Theorem 
1.3]{MP} proved that there exists a universal constant $\delta > 0$ 
such that whenever $A$, $B$, $C$ are normal 
subsets in a finite simple group of Lie type 
$G$, each of size at least $|G|^{1-\delta}$, 
then $ABC = G$. Theorem \ref{main} is an 
improvement of this result.
Theorem \ref{main} is also related to a theorem of 
Gowers. See the next section. 

\section{A theorem of Gowers and the Gowers trick}

Let $G$ be a finite group and let $A$, $B$, $C$ be subsets of 
$G$. As in the Introduction, let $\mathrm{Prob}(A,B,C)$ be the 
probability that if $a$ and $b$ are uniformly 
and randomly chosen elements from $A$ and $B$ 
respectively, then $ab \in C$. Let $k$ be the minimal degree of a non-trivial complex irreducible character of $G$. Gowers proved the following stronger form of 
\cite[Theorem 3.3]{Gowers}, which is implicit in 
its proof and which may be considered as the 
main result of \cite{Gowers}. 

\begin{theorem}[Gowers]
\label{Gowers}	
If $\eta > 0$ is 
such that $|A||B||C| > |G|^{3}/\eta^{2}k$, then 
$$(1 - \eta) \frac{|C|}{|G|} < \mathrm{Prob}
(A,B,C) < (1 + \eta) \frac{|C|}{|G|}.$$ 
\end{theorem}

The Gowers trick was obtained by Nikolov and Pyber \cite[Corollary 1]{NP}. We state it in the 
following form. 

\begin{theorem}
\label{Gowerstrick}	
If $\eta > 0$ is such that 
$|A||B||C| > |G|^{3}/\eta^{2}k$, then for any $g 
\in G$, the number $N$ of triples $(a,b,c) \in A 
\times B \times C$ such that $abc = g$ satisfies 
$$(1 - \eta) \frac{|A||B||C|}{|G|} < N < (1 + 
\eta) \frac{|A||B||C|}{|G|}.$$
\end{theorem}

In the next paragraph we will show that, in the statement of 
Theorem \ref{main}, we may assume that $G$ is a classical simple group 
$\mathrm{Cl}(n,q)$ where $n$ is the dimension of the 
natural module for the lift of $G$ over the field of 
size $q$ unless $G$ is a unitary group when the field 
has size $q^2$. Furthermore, we will show that we may also assume that in the statement of Theorem \ref{main} this $n$ is sufficiently large.

Let $G$ be a finite simple group of Lie type of rank $r$. We have $k > |G|^{1/8r^{2}}$ by 
\cite[Proposition 2.3]{GPSz}. Choose $\delta$ to be less than $1/8r^{2}$. In this case $k > 
|G|^{\delta}$ and so $|G|^{3-\delta}/\eta^2 > |G|^{3}/\eta^2k$ 
for any given $\eta>0$. Thus Theorem \ref{main}
follows from Theorem \ref{Gowers} and Theorem \ref{Gowerstrick}
when $r$ is bounded. Therefore we may assume that $r$ is unbounded, that is, $G$ is a finite simple classical group $\mathrm{Cl}(n,q)$, where $n$ is unbounded. 




\section{Sets permuted}

The aim of this section is to reduce the proof of Theorem \ref{main} to the case when $A$ and $B$ are normal in $G$ (see Proposition \ref{prop:trick}).

Let $G$ be an arbitrary group. For 
arbitrary subsets $X$, $Y$, $Z$ of $G$, 
let $\mathcal{N}(X,Y,Z)$ be $\{(x,y)\in 
X\times Y\vert xy\in Z\}$ and let $X^{-1}=\{x^{-1}\vert x\in X\}$.

\begin{lemma}\label{l:rem_cases_1}
For arbitrary subsets $X$, $Y$, $Z$ of a group $G$, the three sets $\mathcal{N}(X,Y,Z)$, $\mathcal{N}(Y,Z^{-1},X^{-1})$, $\mathcal{N}(Z^{-1},X,Y^{-1})$ have the same cardinality.
\end{lemma}
\begin{proof}
Let $\phi_{1}$ be the map from 
the set $\mathcal{N}(X,Y,Z)$ to 
the set $\mathcal{N}(Y,Z^{-1},X^{-1})$ defined 
by $\phi_{1}(x,y)=(y,(xy)^{-1})$, for 
every $(x,y)\in \mathcal{N}(X,Y,Z)$. 
Let $\phi_{2}$ be the map from 
$\mathcal{N}(Y,Z^{-1},X^{-1})$ to 
$\mathcal{N}(X,Y,Z)$ defined by 
$\phi_{2}(y,z^{-1})=(zy^{-1},y)$, for 
every $(y,z^{-1})\in 
\mathcal{N}(Y,Z^{-1},X^{-1})$. We claim 
that both $\phi_{1}$ and $\phi_{2}$ are 
bijections and that they are inverses of 
one another. For this it is sufficient to 
see that the maps $\phi_{2} \circ \phi_{1}
$ and $\phi_{1} \circ \phi_{2}$ are the 
identity maps on $\mathcal{N}(X,Y,Z)$ and 
on $\mathcal{N}(Y,Z^{-1},X^{-1})$ 
respectively. Indeed, for arbitrary $(x,y) 
\in \mathcal{N}(X,Y,Z)$, we have $$
(\phi_{2} \circ \phi_{1})(x,y) = \phi_{2}
(\phi_{1}(x,y)) = \phi_{2}((y,(xy)^{-1})) 
= ((xy)y^{-1},y) = (x,y)$$ and for 
arbitrary $(y,z^{-1}) \in \mathcal{N}
(Y,Z^{-1},X^{-1})$, we have $$(\phi_{1} 
\circ \phi_{2})(y,z^{-1}) = \phi_{1}
(\phi_{2}(y,z^{-1})) = \phi_{1}
((zy^{-1},y)) = (y,(zy^{-1}y)^{-1}) = 
(y,z^{-1}).$$ 
This shows that $\mathcal{N}(X,Y,Z)$ and 
$\mathcal{N}(Y,Z^{-1},X^{-1})$ are in 
bijection. 

Finally, to prove that $\mathcal{N}(Y,Z^{-1},X^{-1})$ is in bijection with 
$\mathcal{N}(Z^{-1},X,Y^{-1})$, it is enough to repeat the 
argument above with $(Y,Z^{-1},X^{-1})$ in place of $(X,Y,Z)$. 
\end{proof} 
A consequence of Lemma \ref{l:rem_cases_1} is the following.
\begin{corollary}
\label{3.2}	
Let $G$ be a finite group and let $A,B,C$ be non-empty subsets of $G$. Then 
\begin{equation}\label{eq:N}
N(B,C^{-1},A^{-1})=N(A,B,C)=
N(C^{-1},A,B^{-1})
\end{equation}
and 
\begin{equation}\label{eq:last}
\frac{|C|}{|A|}\cdot \mathrm{Prob}(B,C^{-1},A^{-1}) =
\mathrm{Prob}(A,B,C) =
\frac{|C|}{|B|}\cdot \mathrm{Prob}(C^{-1},A,B^{-1}).
\end{equation}
\end{corollary}
\begin{proof}
Recall that  for arbitrary non-empty subsets $X$, $Y$, $Z$ in
a finite group $G$, we defined $N(X,Y,Z)$ to be $|\mathcal{N}(X,Y,Z)|$ and $\mathrm{Prob}(X,Y,Z)$ to be $N(X,Y,Z)/|X||Y|$.
Conclusion (\ref{eq:N}) is a direct consequence of Lemma \ref{l:rem_cases_1} and (\ref{eq:last}) follows from (\ref{eq:N}).
\end{proof}
We introduce some more notation. Fix $g\in G$. For subsets $X,Y,Z$ of $G$, set
$$\mathcal{N}(X,Y,Z,g)=\{(x,y,z)\in X\times Y\times Z\vert xyz=g\}.$$
\begin{lemma}\label{l:rem_cases_2}
Let $G$ be a group, let $X,Y,Z$ be  subsets of $G$ and let $g\in G$. Let $Z$ be normal in $G$. The following hold.
\begin{itemize}
\item[(i)] The sets $\mathcal{N}(X,Y,Z,g)$ and $\mathcal{N}(X,Z,Y,g)$ 
have the same cardinality.
\item[(ii)] If $Y$ is a normal subset in $G$, then the sets $\mathcal{N}(X,Y,Z,g)$ and $\mathcal{N}(Y,Z,X,g)$ have the same cardinality.
\end{itemize}
\end{lemma}
\begin{proof}
(i) Let $\eta_1$ be the map from the set 
$\mathcal{N}(X,Y,Z,g)$ to the set $\mathcal{N}(X,Z,Y,g)$ defined by
$\eta_1(x,y,z)=(x,yzy^{-1},y)$ for every
$(x,y,z)\in \mathcal{N}(X,Y,Z,g)$ and let 
$\eta_2$ be the map from $\mathcal{N}(X,Z,Y,g)$ to $\mathcal{N}
(X,Y,Z,g)$ defined by
$\eta_2(x,z,y)=(x,y,y^{-1}zy)$ for every
$(x,z,y)\in \mathcal{N}(X,Z,Y,g)$. We claim that 
$\eta_2\circ\eta_1$ is the identity map on 
$\mathcal{N}(X,Y,Z,g)$ and that $\eta_1\circ\eta_2$ is the identity 
map on $\mathcal{N}(X,Z,Y,g)$.
For arbitrary $(x,y,z) 
\in \mathcal{N}(X,Y,Z,g)$, we have 
\begin{align*}
(\eta_{2} \circ \eta_{1})(x,y,z) & = \eta_{2}
(\eta_{1}(x,y,z)) = \eta_{2}((x,yzy^{-1},y)) \\ 
& = (x,y,y^{-1}(yzy^{-1})y) = (x,y,z)
\end{align*}
and for 
arbitrary $(x,z,y) \in \mathcal{N}
(X,Z,Y,g)$, we have 
\begin{align*}
(\eta_{1} 
\circ \eta_{2})(x,z,y) & = \eta_{1}
(\eta_{2}(x,z,y))  = \eta_{1}
((x,y,y^{-1}zy)) \\
&= (x,y(y^{-1}zy)y^{-1},y) = 
(x,z,y).
\end{align*}

(ii) Let $\theta_1$ be the map from the set  $\mathcal{N}(X,Y,Z,g)$ to the set $\mathcal{N}(Y,Z,X,g)$ defined by
$\theta_1(x,y,z)=(xyx^{-1},xzx^{-1},x)$ for every
$(x,y,z)\in \mathcal{N}(X,Y,Z,g)$ and let 
$\theta_2$ be the map from the set $\mathcal{N}(Y,Z,X,g)$ to 
the set $\mathcal{N}(X,Y,Z,g)$ defined by
$\theta_2(y,z,x)=(x,x^{-1}yx,x^{-1}zx)$ for every
$(y,z,x)\in \mathcal{N}(Y,Z,X,g)$. We claim that 
$\theta_2\circ\theta_1$ is the identity map on 
the set $\mathcal{N}(X,Y,Z,g)$ and that $\theta_1\circ\theta_2$ 
is the 
identity map on the set $\mathcal{N}(Y,Z,X,g)$.
For arbitrary $(x,y,z) 
\in \mathcal{N}(X,Y,Z,g)$, we have 
\begin{align*}
(\theta_{2} \circ \theta_{1})(x,y,z) & = \theta_{2}
(\theta_{1}(x,y,z)) = \theta_{2}((xyx^{-1},xzx^{-1},x)) \\ 
& = (x,x^{-1}(xyx^{-1})x,x^{-1}(xzx^{-1})x) = (x,y,z)
\end{align*}
and for 
arbitrary $(y,z,x) \in \mathcal{N}
(Y,Z,X,g)$, we have 
\begin{align*}
(\theta_{1} 
\circ \theta_{2})(y,z,x) & = \theta_{1}
(\theta_{2}(y,z,x))  = \theta_{1}
((x,x^{-1}yx,x^{-1}zx)) \\
&= (x(x^{-1}yx)x^{-1}, x(x^{-1}zx)x^{-1},x ) = 
(y,z,x).
\end{align*}
\end{proof}
Note that if $G$ is finite, then  
$|\mathcal{N}(X,Y,Z,g)|=N(X,Y,gZ^{-1})$.

\begin{proposition}\label{prop:trick}
If Theorem \ref{main} is true in the special case when $A$ and $B$ are normal, then Theorem \ref{main} is true in general. 
\end{proposition}

\begin{proof}
Let $A$, $B$, $C$ be subsets of $G$ satisfying conditions (1) and (2) of Theorem \ref{main}.  We have two cases to consider: (i) $A$ and $C$ are normal in $G$ and (ii) $B$ and $C$ are normal in $G$. Observe that if $X$ is a normal set in $G$ then $X^{-1}$ is also normal in $G$.  

If $A$ and $C$ are normal in $G$, then our hypothesis gives 
\begin{equation}
\label{AC}
(1-\eta)\frac{|B|}{|G|}<  \mathrm{Prob}(C^{-1},A,B^{-1})< (1+\eta)\frac{|B|}{|G|}.
\end{equation}
Applying (\ref{eq:last}), we deduce that 
$$(1-\eta)\frac{|C|}{|G|} < \mathrm{Prob}(A,B,C) < (1+\eta)\frac{|C|}{|G|},$$
which is the first conclusion of Theorem \ref{main}. 
Fix  $g$ in $G$. Let 
$N=|\mathcal{N}(A,B,C,g)|$. This is equal to $N(A,B,gC^{-1})$. 
By applying our hypothesis to the triple $(A,C,B)$, we deduce that  
$$(1-\eta)\frac{|A||B||C|}{|G|} <  |\mathcal{N}(A,C,B,g)| <  (1+\eta)\frac{|A||B||C|}{|G|}.$$
But
$$
|\mathcal{N}(A,C,B,g)|
=|\mathcal{N}(A,B,C,g)|=N$$
 by Lemma \ref{l:rem_cases_2},  
and this proves the second conclusion of Theorem \ref{main} in this special case.

If $B$ and $C$ are normal in $G$, then 
by applying our hypothesis to the triple $(B,C^{-1},A^{-1})$ in place of $(A,B,C)$, we deduce that
\begin{equation}
\label{BC}	
 (1-\eta)\frac{|A|}{|G|} <   \mathrm{Prob}(B,C^{-1},A^{-1}) <  (1+\eta)\frac{|A|}{|G|}.
\end{equation}
We get 
$$(1-\eta)\frac{|C|}{|G|} < \mathrm{Prob}(A,B,C) < (1+\eta)\frac{|C|}{|G|},$$ 
by applying (\ref{eq:last}).
Fix  $g\in G$. 
Our hypothesis for the triple $(B,C,A)$ implies that 
$$(1-\eta)\frac{|A||B||C|}{|G|} <  |\mathcal{N}(B,C,A,g)|  <  (1+\eta)\frac{|A||B||C|}{|G|}.$$
But 
$$
|\mathcal{N}(B,C,A,g)|
=|\mathcal{N}(A,B,C,g)|=N$$
by Lemma \ref{l:rem_cases_2}, 
and this proves the second conclusion of Theorem \ref{main} in this special case too. 
\end{proof}

From now on, in order to prove our main result, we may assume that in the statement of Theorem \ref{main}, $A$ and $B$ are normal.

\section{The second conclusion of Theorem \ref{main}}

We claim that the second conclusion of Theorem \ref{main} follows from the first. For this we may assume that $A$ and $B$ are normal in $G$. Fix $g \in G$. The number $N$ of triples $(a,b,c) \in A \times B \times C$ such that $abc = g$ is equal to $N(A,B,gC^{-1})$. Observe that $|gC^{-1}| = |C|$ (and $C$ need not be normal). We get $$(1 - \eta) \frac{|C|}{|G|} < \mathrm{Prob}(A,B,gC^{-1}) < (1 + \eta) \frac{|C|}{|G|}$$ by the first conclusion. The second conclusion now follows from the fact that $\mathrm{Prob}(A,B,gC^{-1}) = N(A,B,gC^{-1})/|A||B|$. 

From now on, we focus on the first conclusion of Theorem \ref{main}.

\section{Changing Hypothesis (2)}
We will show that we may replace Hypothesis (2) of Theorem \ref{main} by (2') below. Let $A$, $B$, $C$ be subsets in $G$. Let $\eta > 0$ and let $\delta > 0$ be as in the statement of Theorem \ref{main}. Hypothesis (2) of Theorem \ref{main} states that $|A||B||C|$ is larger than  $|G|^{3-\delta}/\eta^2$. This implies that $|A|$, $|B|$, $|C|$ are larger than  $|G|^{1-\delta}/\eta^2$. On the other hand, if $|A|$, $|B|$, $|C|$ are larger than  $|G|^{1-(\delta/3)}/\eta^2$, then Hypothesis (2) of Theorem \ref{main} holds. By changing $\delta$ to $\delta/3$, in the rest of the paper we will replace Hypothesis (2) by the following. 

\begin{itemize}
\label{(2')}	
\item[(2')] The subsets $A$, $B$, $C$ have size larger than $|G|^{1-\delta}/\eta^2$.  
\end{itemize}

\section{Three conjugacy classes}

We will prove Theorem \ref{main} in the case when $A$, $B$, $C$ are conjugacy classes. 

Let $G$ be a finite group and let $\mathrm{Irr}(G)$ be the set of complex irreducible characters of $G$. For an element $g \in G$ and a character $\chi \in \mathrm{Irr}(G)$, it is useful to bound $|\chi(g)|$ in terms of a fixed power of $\chi(1)$. Such character bounds were first used in the fundamental paper by Diaconis and Shahshahani \cite{DS} where they were applied to random walks on symmetric groups. The following is a special case of an important theorem of Guralnick, Larsen, Tiep \cite[Theorem 1.3]{GLT}. 

\begin{theorem}[Guralnick, Larsen, Tiep]
	\label{GLT}
	There exists a universal constant 
	$\mu > 0$ such that whenever $G$ is a 
	classical simple group and $g \in G$ satisfies 
	$|C_{G}(g)| \leq {|G|}^{\mu}$, then 
	$|\chi(g)| \leq \chi(1)^{1/10}$ for every $\chi 
	\in \mathrm{Irr}(G)$. 
\end{theorem}


Let $A$, $B$, $C$ be conjugacy classes of a finite group $G$ and let $a$, $b$, $c$ be representatives in $A$, $B$, $C$ respectively. We have 
\begin{equation}
	\label{e2}	 
	N(A,B,C) = \frac{|A||B||C|}{|G|} 
	\sum_{\chi \in \mathrm{Irr}(G)} 
	\frac{\chi(a) \chi(b) \overline{\chi(c)}}
	{\chi(1)}
\end{equation}
by \cite[p. 43-44]{ASH}. 

For any positive number $x$, the well-known Witten zeta function $\zeta^{G}(x)$ is defined to be $\sum_{\chi \in \mathrm{Irr}(G)} \chi(1)^{-x}$. A special case of an important theorem of Liebeck and Shalev \cite[Theorem 1.1]{LS2} is the following.

\begin{theorem}[Liebeck, Shalev]\label{zeta}
	For any sequence of non-abelian finite 
	simple groups $G \not= \mathrm{PSL}(2,q)$ (for any prime power $q$) and 
	any $x > 2/3$, $\zeta^{G}(x) \to 1$ as $|G| \to 
	\infty$.
\end{theorem}

We are now in the position to prove Theorem \ref{main} in the special case when the sets $A$, $B$, $C$ are conjugacy classes in $G$. For this, we may assume that $G$ is a classical simple group $\mathrm{Cl}(n, q)$ where $n$ is sufficiently large and we may replace Hypothesis (2) by (2').  

\begin{theorem}\label{p:G_classical}
Let $G$ be a classical simple group $\mathrm{Cl}(n, q)$. Fix $\eta > 0$. There is a $\delta$ with $0 < \delta < 1$ such that whenever $A$, $B$, $C$ are conjugacy classes of $G$ each of size larger than $|G|^{1-\delta}/\eta^{2}$, then 
$$(1-\eta)\frac{|C|}{|G|} < \mathrm{Prob}(A,B,C) < (1+\eta)\frac{|C|}{|G|}.$$ 
\end{theorem}

\begin{proof}
We may choose $n$ large enough by the last paragraph of Section 2. 
	Let $\mu$ be as in Theorem \ref{GLT}. 
	Let $A$, $B$, $C$ be conjugacy classes of $G$ 
	each of size larger than $|G|^{1-\delta}/\eta^{2}>|G|^{1-\mu}$ for some $\delta$. As $n$ may be chosen large enough, we may assume that $\zeta^{G}(7/10) -1 < \eta$ by Theorem \ref{zeta}. Let $a\in A$, $b\in B$ and $c\in C$. We have by (\ref{e2}) that  
	\begin{align*}
		\left|N(A,B,C) - \frac{|A||B||C|}{|G|}\right|& = 
		\frac{|A||B||C|}{|G|}\left|\sum_{1\neq \chi \in \mathrm{Irr}(G)} 
		\frac{\chi(a) \chi(b) \overline{\chi(c)}}
		{\chi(1)}\right|\\
		&\leq \frac{|A||B||C|}{|G|} \sum_{1\neq \chi \in \mathrm{Irr}(G)} 
		\frac{|\chi(a)| |\chi(b)| |\overline{\chi(c)}|}
		{|\chi(1)|}\\
		&\leq \frac{|A||B||C|}{|G|}\sum_{1\neq \chi \in \mathrm{Irr}(G)} 
		\chi(1)^{-7/10}\\
		&=\frac{|A||B||C|}{|G|}(\zeta^{G}(7/10) -1 )\\
		&< \frac{|A||B||C|}{|G|}\eta. 
	\end{align*}
The result follows.
	\end{proof}	

\section{Three normal sets}

We will prove Theorem \ref{main} in the case when the subsets $A$, $B$, $C$ are normal. 

\begin{lemma}\label{k(G)}
Let $G = \mathrm{Cl}(n,q)$. There exists a universal constant $c$ such that $$k(G) \leq |G|^{c/n}.$$	
\end{lemma}

\begin{proof}
	We have $k(G) \leq q^{c_{1} n}$ for some universal constant $c_{1}$ by \cite[Theorem 1.1]{LP} (see also \cite[Corollary 1.2]{FG}). By the order formulas for finite simple classical groups, there exists a universal constant $c_{2} > 0$ such that $|G| = |\mathrm{Cl}(n,q)| \geq q^{c_{2}n^{2}}$. We get 
	$$k(G) \leq q^{c_{1} n} = {(q^{c_{2}n^2})}^{c_{1}/(c_{2}n)} \leq |G|^{c/n},$$ where $c=c_{1}/c_{2}$.
\end{proof}

\begin{lemma}\label{l:norm->conj}
Let $G = \mathrm{Cl}(n,q)$. Fix $\eta > 0$ and $\delta > 0$. Let $X$ be a normal subset of $G$ with $|X| > |G|^{1-\delta}/\eta^2$. For any fixed $\alpha > \delta$, the set $X$ contains a conjugacy class $Y$ of $G$ with $|Y| > |G|^{1-\alpha}/\eta^2$, provided that $n$ is sufficiently large. 
\end{lemma}

\begin{proof}
If no such conjugacy class $Y$ of $G$ is contained in the normal subset $X$ of $G$, then $$|G|^{1-\delta}/\eta^2 < |X| \leq k(G) |G|^{1-\alpha}/\eta^2 \leq |G|^{1 + (c/n) -\alpha}/\eta^2$$ by Lemma \ref{k(G)}. Thus $c/n > \alpha - \delta$. This is a contradiction since $c/n$ tends to $0$ as $n$ goes to infinity.  	
\end{proof}

Let $G = \mathrm{Cl}(n,q)$. Fix $\eta$ with  $0<\eta<1$. Let $\delta > 0$ later to be specified. Let $A$, $B$, $C$ be normal subsets of $G$ each of size larger than $|G|^{1-\delta}/\eta^2$. Let $X \in \{ A, B, C \}$. Let $X_1$ be the union of all conjugacy classes in $G$ which are contained in $X$ and each of which have size larger than $|G|^{1-\alpha}/\eta^2$ for some fixed $\alpha > \delta$ soon to be determined (in the end of the proof we will require $\delta>0$ to be small and $\alpha>0$ such that $\alpha>3\delta$). Let us call such conjugacy classes large. Since $n$ may be taken to be sufficiently large, the normal set $X_1$ is non-empty by Lemma \ref{l:norm->conj}. 

Let $K_{a_1}, \ldots , K_{a_r}$ be the list of (distinct) large conjugacy classes of $G$ contained in $A_1$. Similarly, let $K_{b_1}, \ldots , K_{b_s}$ be the list of large conjugacy classes of $G$ contained in $B_1$, and let $K_{c_1}, \ldots , K_{c_t}$ be the list of large conjugacy classes of $G$ contained in $C_1$. Let $a_i$, $b_j$, $c_l$ be fixed indices such that $1 \leq i \leq r$, $1 \leq j \leq s$, $1 \leq l \leq t$. There is a choice of $\delta>0$ in  Theorem \ref{p:G_classical} with $\eta/2$ such that 
$$(1-(\eta/2))\frac{|K_{a_i}||K_{b_{j}}||K_{c_l}|}{|G|} < N(K_{a_i}, K_{b_j}, K_{c_{l}}) < (1+(\eta/2))\frac{|K_{a_i}||K_{b_{j}}||K_{c_l}|}{|G|}.$$ 
 This immediately implies that 
$$(1-(\eta/2)) \frac{|A_1||B_1||C_1|}{|G|} < \sum_{i=1}^{r} \sum_{j=1}^{s} \sum_{l=1}^{t} N(K_{a_{i}}, K_{b_{j}}, K_{c_{l}}) < (1+(\eta/2))\frac{|A_1||B_1||C_1|}{|G|}.$$ Since 
$$N(A_{1},B_{1},C_{1}) = \sum_{i=1}^{r} \sum_{j=1}^{s} \sum_{l=1}^{t} N(K_{a_{i}}, K_{b_{j}}, K_{c_{l}}),$$ it follows that 
\begin{equation}
\label{threenormal}
(1-(\eta/2)) \frac{|A_1||B_1||C_1|}{|G|} < N(A_{1},B_{1},C_{1}) < (1+(\eta/2)) \frac{|A_1||B_1||C_1|}{|G|}. 	
\end{equation}
 
For $X_{2} = X \setminus X_{1}$, we have, by Lemma \ref{k(G)}, that 
\begin{equation}
\label{X2}	
|X_{2}| \leq k(G) |G|^{1-\alpha}/\eta^2 \leq |G|^{1 + (c/n) -\alpha}/\eta^2 \leq \beta|G|^{1-\delta}/\eta^2 < \beta |X|
\end{equation}
for any fixed $\beta > 0$, provided that $n$ is sufficiently large. 
It follows that 
\begin{equation}
\label{X1}	
|X_{1}| > (1-\beta)|X|.
\end{equation}     
Let $i, j, l \in \{1,2 \}$. Observe that $N(A_i,B_j,C_l) \leq |G| \min\{ |A_i|,|B_j|,|C_l| \}$. We have
$$N(A,B,C) = \sum_{i=1, j=1, l=1}^2 N(A_i,B_j,C_l) \leq N(A_1,B_1,C_1)+7|G|\max\{
|A_2|,|B_2|,|C_2|\}.$$ Since $7|G|\max\{ |A_2|,|B_2|,|C_2| \} \leq 7|G|^{2 + (c/n) -\alpha}/\eta^2$ by (\ref{X2}), it follows from this that
\begin{equation}
\label{N(A,B,C)}	
N(A_1,B_1,C_1) \leq N(A,B,C) \leq N(A_1,B_1,C_1) + 7|G|^{2 +  (c/n)  -\alpha}/\eta^2.
\end{equation}
Formulas (\ref{N(A,B,C)}), (\ref{threenormal}), and (\ref{X1}) give $$N(A,B,C) \geq N(A_{1},B_{1},C_{1}) > (1-(\eta/2)) \frac{|A_1||B_1||C_1|}{|G|} > (1-(\eta/2)) (1-\beta)^{3} \frac{|A||B||C|}{|G|}.$$ For $\beta < 1 - {(2(1-\eta)/(2-\eta))}^{1/3}$, we have $(1-(\eta/2)) (1-\beta)^{3} > 1-\eta$, that is, 
\begin{equation}
\label{star}	
N(A,B,C) > (1-\eta) \frac{|A||B||C|}{|G|}.
\end{equation}
On the other hand, (\ref{N(A,B,C)}) and (\ref{threenormal}) provide
\begin{equation}
\label{equation1}
N(A,B,C) < (1+(\eta/2)) \frac{|A||B||C|}{|G|} + 7|G|^{2 +  (c/n)  -\alpha}/\eta^2.
\end{equation}
Now 
\begin{equation}
\label{equation2}	
7|G|^{2 +  (c/n)  -\alpha}/\eta^2 \leq |G|^{2-3\delta}/(2\eta) \leq (\eta/2) \frac{|A||B||C|}{|G|}, 
\end{equation}
provided that $\alpha > 3 \delta$ and $n$ is sufficiently large. Formulas (\ref{equation1}) and (\ref{equation2}) give 
\begin{equation}
\label{starstar}	
N(A,B,C) < (1+\eta) \frac{|A||B||C|}{|G|}.  
\end{equation}
Finally, (\ref{star}) and (\ref{starstar}) provide (the first conclusion of) Theorem \ref{main} in the case when $A$, $B$, $C$ are normal subsets in $G$. 

\section{Product mixing}

For positive numbers $\epsilon$ and $\eta$, Lifshitz and Marmor \cite[Section 2.3]{LM} defined a finite group $G$ to be an $(\epsilon, \eta)$-mixer if for all subsets $A$, $B$, $C$ of $G$ with $|A|$, $|B|$, $|C|$ all at least $\epsilon |G|$, we have $$(1 - \eta) \frac{|C|}{|G|} < \mathrm{Prob}(A,B,C) < (1 + \eta) \frac{|C|}{|G|}.$$ They also say that $G$ is normally an $(\epsilon, \eta)$-mixer if the same holds for all such normal subsets $A$, $B$, $C$. (We remark that these properties were defined for $\eta = 0.01$.) Theorem \ref{Gowers} implies that the alternating group $A_n$ is an $(\epsilon, \eta)$-mixer for $\epsilon = C n^{-1/3}$ where $C = C(\eta)$ is a constant depending only on $\eta$. For normal subsets, this result was improved exponentially. The following may be found in \cite[Theorem 2.5]{LM}. 

\begin{theorem}[Lifshitz, Marmor]
	\label{LM1}	
	For any $\eta > 0$, there exists an absolute constant $c > 0$, such that $A_n$ is normally an
	$(n^{-c n^{1/3}}, \eta)$-mixer.
\end{theorem}

It is shown in \cite[Theorem 8.1]{LM} that Theorem \ref{LM1} is best possible in the sense that there exists an absolute constant $C$ (depending on $\eta$) such that $A_n$ is not normally an $(n^{-Cn^{1/3}}, \eta)$-mixer. 

It would be interesting to extend Theorem \ref{LM1} in the spirit of Theorem \ref{main}, however with our current method this is not possible. 

In the rest of the paper we will work with the following definition. 

\begin{definition}\label{def:mixer}
	Let $\epsilon$ and $\eta$ be positive real numbers 
	less than $1$. Let $i \in \{ 1, 2, 3 \}$. The finite group 
	$G$ is an $(\epsilon, \eta, i)$-mixer if whenever $A$, 
	$B$, $C$ are subsets of $G$ each of size at least $
	\epsilon |G|$ and $i$ of these subsets are normal in 
	$G$, then $$(1-\eta) \frac{|C|}{|G|} < \mathrm{Prob}
	(A,B,C) < (1 + \eta) \frac{|C|}{|G|}.$$
\end{definition}

For a positive real number $\epsilon$ less than $1$ and for a finite group $G$, let $k_{\epsilon}(G)\geq 1$ denote the number of conjugacy classes $K$ of $G$ such that $|K| < \epsilon |G|$.

\begin{proposition}\label{prop:epsilon'}
Let $ \eta $ and $\epsilon $ be positive real numbers satisfying the inequalities  $\eta<1/2$ and 
$\epsilon < \min \{ 1, \eta \cdot {k_{\epsilon}(G)}^{-1} {(1-\eta)}^{-2} \}$. Let $G$ be a finite group which is an $(\epsilon, \eta, 3)$-mixer. Let $\epsilon' = {(\epsilon \cdot k_{\epsilon}(G) / \eta)}^{1/2} < 1$. If $A$, $B$, $C$ are subsets of $G$ each of size at least $\epsilon' |G|$ with $A$ and $B$ normal in $G$, then 
$$(1 - 2\eta) \frac{|C|}{|G|} < \mathrm{Prob}(A, B, C) < (1 + 2\eta) \frac{|C|}{|G|}.$$ 
\end{proposition}

\begin{proof} 
Let $A$, $B$, $C$ be subsets of $G$ each of size at least $\epsilon' |G|$ with $A$ and $B$ normal in $G$. Since $N(A,B,C) = \sum_{c \in C} N(A,B,\{ c \})$, we have 
\begin{equation}
\label{m1}
\mathrm{Prob}(A,B,C) = \frac{1}{|A||B|} \sum_{c \in C} N(A,B,\{ c \}).
\end{equation}
Let $m$ be the number of conjugacy classes of $G$. Let the list of conjugacy classes of $G$ be $K_{1}, \ldots , K_{m}$ arranged in such a way that the conjugacy classes $K_{1}, \ldots , K_{t}$ have sizes at least $\epsilon |G|$ and the conjugacy classes $K_{t+1}, \ldots , K_{m}$ have sizes less than $\epsilon |G|$. Let $K$ be the union of the conjugacy classes $K_{t+1}, \ldots , K_{m}$. For each $i \in \{ 1, \ldots , m \}$, let $c_{i}$ be an element from $K_{i}$.  

Since $A$ and $B$ are normal in $G$, the number $N(A,B,\{ c_i \})$ is independent from the choice of $c_i$ in $K_i$. This gives
\begin{equation}
\label{m2}	
\sum_{c \in C} N(A,B,\{ c \}) = \sum_{i=1}^{m} |C \cap K_{i}| \cdot N(A,B,\{c_{i}\}) =
\sum_{i=1}^{m} |C \cap K_{i}| \cdot \frac{N(A,B,K_{i})}{|K_{i}|}.
\end{equation} 
From (\ref{m1}) and (\ref{m2}) we get
\begin{align}	
\mathrm{Prob}(A,B,C) & = \frac{1}{|A||B|} \Big( \sum_{i=1}^{m} |C \cap K_{i}| \cdot \frac{N(A,B,K_{i})}{|K_{i}|} \Big) = \nonumber \\ 
&= \frac{1}{|A||B|} \Big( \sum_{i=1}^{t} |C \cap K_{i}| \cdot \frac{N(A,B,K_{i})}{|K_{i}|}  + \sum_{i=t+1}^{m} |C \cap K_{i}| \cdot \frac{N(A,B,K_{i})}{|K_{i}|} \Big). \label{m3}
\end{align} 
Since $N(A,B,K_{i}) \leq |A||K_{i}|$ for every $i$ in $\{ 1, \ldots , m \}$ and $|B|, |C| \geq \epsilon' |G|$, we have 
\begin{align}
\frac{1}{|A||B|} \sum_{i=t+1}^{m} |C \cap K_{i}| \cdot \frac{N(A,B,K_{i})}{|K_{i}|} &  \leq \frac{1}{|B|} \sum_{i=t+1}^{m} |C \cap K_{i}| \nonumber\\
& \leq \frac{|C \cap K|}{|B|} \leq \frac{|C \cap K|}{\epsilon' |G|}  \leq \frac{|K|}{\epsilon' |G|} \nonumber \\ 
& \leq \frac{k_{\epsilon}(G) \epsilon|G|}{\epsilon' |G|} \nonumber = k_{\epsilon}(G) (\epsilon / \epsilon')  = \eta \epsilon' \nonumber \\
&\leq \eta \frac{|C|}{|G|}. \label{m4}	
\end{align}
Formulas (\ref{m3}) and (\ref{m4}) give 
\begin{equation}
\begin{split}	
\label{m5}	
0 \leq \mathrm{Prob}(A,B,C) - \Big( \sum_{i=1}^{t} \frac{|C \cap K_{i}|}{|K_{i}|} \cdot \mathrm{Prob}(A,B,K_{i}) \Big) \leq \eta \frac{|C|}{|G|}.
\end{split}
\end{equation}
Observe that $\epsilon' \geq \epsilon$ (since $k_\epsilon(G)\geq 1>\eta/(1-\eta)$). Since $G$ is an 
$(\epsilon, \eta, 3)$-mixer, we have 
\begin{equation}
\label{m6}	
(1 - \eta) \frac{|K_{i}|}{|G|} < \mathrm{Prob}(A,B,K_{i}) < (1 + \eta) \frac{|K_{i}|}{|G|}
\end{equation}
for every $i \in \{ 1, \ldots , t \}$. Inequalities (\ref{m5}) and (\ref{m6}) give the required upper bound
\begin{align*}
\mathrm{Prob}(A,B,C) & < (1+\eta) \Big( \sum_{i=1}^{t} \frac{|C \cap K_{i}|}{|G|} \Big) + \eta \frac{|C|}{|G|}	\\
& = (1 + \eta) \frac{|C \cap (G \setminus K)|}{|G|} + \eta \frac{|C|}{|G|} \\ & \leq (1 + 2\eta) \frac{|C|}{|G|}.  
\end{align*} 
Inequalities (\ref{m5}) and (\ref{m6}) also give
\begin{align}
\mathrm{Prob}(A,B,C) & \geq \sum_{i=1}^{t} \frac{|C \cap K_{i}|}{|K_{i}|} \cdot \mathrm{Prob}(A,B,K_{i})\nonumber \\
& > (1-\eta) \sum_{i = 1}^{t} \frac{|C \cap K_{i}|}{|G|} \nonumber  \\
& = (1-\eta) \frac{|C \cap (G \setminus K)|}{|G|}\nonumber \\ 
& \geq (1 - \eta) \Big( \frac{|C| - |K|}{|G|} \Big). \label{m7}
\end{align}
Since $|K| \leq k_{\epsilon}(G) \epsilon |G|$, inequality (\ref{m7}) gives 
\begin{equation}
	\label{m8}	
	\begin{split}
		\mathrm{Prob}(A,B,C) > (1 - \eta) \frac{|C|}{|G|} - (1-\eta) \frac{|K|}{|G|} \geq (1 - \eta) \frac{|C|}{|G|} - (1-\eta) k_{\epsilon}(G) \epsilon.
	\end{split}
\end{equation}
Since $|C| \geq \epsilon' |G|$, we have $\eta |C|/|G| \geq \eta \epsilon'$. Since $\epsilon' = {(\epsilon k_{\epsilon}(G)/\eta)}^{1/2}$, we get $\eta |C|/|G| \geq {(\eta \epsilon k_{\epsilon}(G))}^{1/2}$. In view of this and (\ref{m8}), in order to complete the proof of the lemma, it is sufficient to show that ${(\eta \epsilon k_{\epsilon}(G))}^{1/2} \geq (1-\eta) k_{\epsilon}(G) \epsilon$. This inequality is equivalent to the inequality $\epsilon \leq \eta (1-\eta)^{-2} k_{\epsilon}(G)^{-1}$. But this is part of the conditions of our lemma.  
\end{proof}	

We deduce the following consequence of Proposition \ref{prop:epsilon'}. This is not needed for the proof of Theorem \ref{main}.

\begin{theorem}
Let $ \eta $ and $\epsilon $ be positive real numbers satisfying the inequalities  \\ $\eta<1/2$ and 
$\epsilon < \min \{ 1, \eta \cdot {k_{\epsilon}(G)}^{-1} {(1-\eta)}^{-2} \}$. Let $G$ be a finite group which is an $(\epsilon, \eta, 3)$-mixer. Let $\epsilon' = {(\epsilon \cdot k_{\epsilon}(G) / \eta)}^{1/2} < 1$.  If a finite group $G$ is an $(\epsilon, \eta, 3)$-mixer, then it is also an $(\epsilon', 2 \eta, 2)$-mixer.
\end{theorem}

\begin{proof}
Let $G$ be an $(\epsilon, \eta, 3)$-mixer. Let $A$, $B$, $C$ be subsets of $G$ each of size at least $\epsilon'|G|$. Assume that two of the sets $A$, $B$, $C$ are normal in $G$. If $A$ and $B$ are normal in $G$, then the result follows by  Proposition \ref{prop:epsilon'}. Let $A$ and $C$ be normal in $G$. Then 
$$(1-2\eta) \frac{|B^{-1}|}{|G|} < \mathrm{Prob}(C^{-1},A,B^{-1}) < (1+2\eta) \frac{|B^{-1}|}{|G|}$$ by Proposition \ref{prop:epsilon'}. Thus $$(1-2\eta) \frac{|C|}{|G|} < \frac{|C|}{|B|} \cdot \mathrm{Prob}(C^{-1},A,B^{-1}) < (1+2\eta) \frac{|C|}{|G|}.$$ Since $$\frac{|C|}{|B|} \cdot \mathrm{Prob}(C^{-1},A,B^{-1}) = \mathrm{Prob}(A,B,C)$$ by Corollary \ref{3.2}, the result follows. Finally, let $B$ and $C$ be normal in $G$. Then $$(1-2\eta) \frac{|A^{-1}|}{|G|} < \mathrm{Prob}(B,C^{-1},A^{-1}) < (1+2\eta) \frac{|A^{-1}|}{|G|}$$ by Proposition \ref{prop:epsilon'}. Thus $$(1-2\eta) \frac{|C|}{|G|} < \frac{|C|}{|A|} \cdot \mathrm{Prob}(B,C^{-1},A^{-1}) < (1+2\eta) \frac{|C|}{|G|}.$$ Since $$\frac{|C|}{|A|} \cdot \mathrm{Prob}(B,C^{-1},A^{-1}) = \mathrm{Prob}(A,B,C)$$ by Corollary \ref{3.2}, the result follows in this case too. The proof is complete.  
\end{proof}

\section{Proof of Theorem \ref{main}}

In Section 2 we showed that, in order to prove Theorem \ref{main}, we may assume that $G$ is a finite simple classical group $\mathrm{Cl}(n,q)$ with $n$ large enough. Given $\eta$ with  $0<\eta<1/4$ and $\delta > 0$, we may also replace Hypothesis (2) by (2'). In  Section 4 we also showed that it is sufficient to establish the first conclusion of Theorem \ref{main}. We may assume that $A$ and $B$ are normal in $G$ by Proposition \ref{prop:trick}. If $C$ is normal in $G$, Theorem \ref{main} follows from Section 6. In the language of Definition \ref{def:mixer}, $G$ is an $(\epsilon, \eta, 3)$-mixer where $\epsilon = |G|^{-\delta}/\eta^{2}$. By changing $\eta$ to $\eta/2$, we also have that $G$ is an $(\epsilon, \eta/2, 3)$-mixer where $\epsilon = 4 |G|^{-\delta}/\eta^{2}$. Finally, assume that $C$ is not normal in $G$. Observe that $k_\epsilon(G)\geq 1$ since $\epsilon|G|=4|G|^{1-\delta}/\eta^2> 1$ for $n$ large enough.  
We have $k_{\epsilon}(G) \leq k(G) \leq |G|^{c/n}$ by Lemma \ref{k(G)}. It follows that 
$$\epsilon = 4 |G|^{-\delta}/\eta^{2} < \min \{ 1, \eta {(1-\eta)}^{-2} |G|^{-c/n}\},$$ for any given $\delta > 0$, provided that $n$ is sufficiently large. Now $G$ is an $(\epsilon', \eta, 2)$-mixer by Proposition \ref{prop:epsilon'}, where $$\epsilon' = {(\epsilon \cdot k_{\epsilon}(G)/ \eta)}^{1/2} \leq (2/\eta^{3/2}) |G|^{((c/n)-\delta)/2}.$$ This is at most $|G|^{-\delta/3}/\eta^{2}$ provided that $n$ is sufficiently large. In this case the first conclusion of Theorem \ref{main} holds with $\delta/3$ in place of $\delta$.

\end{document}